\documentclass[final,1p,times]{elsarticle}
\usepackage{amssymb}
\usepackage{amsmath}
\usepackage{amsthm}
\usepackage[utf8]{inputenc}
\usepackage{marginnote}
\usepackage{amsfonts}
\usepackage{textcomp}
\usepackage{color,soul}
\usepackage{placeins}
\usepackage{tikz}
\usepackage{makecell}
\usepackage{amssymb}
\usepackage{graphicx}
\usepackage{hyperref}
\usepackage[none]{hyphenat}
\usepackage{lipsum}
\newtheorem{thm}{Theorem}[section]

\newtheorem{rem}[thm]{Remark}
\newtheorem{prop}[thm]{Proposition}
\newtheorem{cor}[thm]{Corollary}
\newtheorem{example}[thm]{Example}
\newdefinition{defn}[thm]{Definition}
\allowdisplaybreaks

\journal{Mathematische Nachrichten}
\begin{document}
\begin{frontmatter}
		\title{ Jensen-type inequalities on pairs of measure spaces and their applications}
        \author{Shankhadeep Mondal}
		\ead{shankhadeep.mondal@ucf.edu}
        \author{P. D. Johnson}
		\ead{johnspd@auburn.edu}
		\author{R. N. Mohapatra}
		\ead{ram.mohapatra@ucf.edu}
        \author{David Guinovart}
		\ead{sguino001@umn.edu}

		 \address{School of Mathematics, University of Central Florida, Orlando, Florida-32816}
		\address{Department of Mathematics and Statistics, Auburn University, Auburn, AL-36849 }
       \address{School of Mathematics, University of Central Florida, Orlando, Florida-32816}
       \address{Hormel Institute, University of Minnesota, Austin, MN 55912}
       
		\cortext[shan]{Corresponding Author:shankhadeep.mondal@ucf.edu}
		
\begin{abstract}
We develop applications of a Jensen-type inequality for pairs of positive measure spaces. The framework yields $L^p$ and operator inequalities for positive integral operators, norm and entropy estimates for Fej\'er means, and probabilistic moment and variance bounds. We also obtain robustness estimates under random and deterministic erasures. These results provide a unified connection between Jensen-type inequalities, harmonic analysis, operator theory, probability, and robust reconstruction.
\end{abstract}


\end{frontmatter}

\section{Introduction and background}

Inequalities involving vertex degrees and average degrees arise naturally in graph and hypergraph theory. For a $k$-uniform hypergraph $G=(V,E)$ with incidence matrix $M=[M(v,e)]$, the $k$-uniformity condition means that every column of $M$ has the same sum $k$. The degree of $v\in V$ is
\[
d(v)=\sum_{e\in E}M(v,e),
\]
and its average degree is
\[
\bar d=\frac{1}{|V|}\sum_{v\in V}d(v).
\]
Johnson and Perry \cite{JP} established the following general matrix
inequality relating such quantities.

\begin{thm}[\cite{JP}]\label{t1}
Let $M=[m_{ij}]$ be a $p\times q$ real matrix with constant column
sum $c$ and row sums $d_1,\ldots,d_p$. Set
$\bar d=\frac{1}{p}\sum_{i=1}^p d_i=\frac{qc}{p}$.
If $f(t)=t\varphi(t)$ is convex on an interval containing
$d_1,\ldots,d_p$, then
\[
\frac{1}{q}\sum_{j=1}^q\sum_{i=1}^p
m_{ij}\varphi(d_i)\geq c\varphi(\bar d).
\]
If $f$ is strictly convex, equality holds if and only if
$d_1=\cdots=d_p$.
\end{thm}

The weighted analogue obtained by Johnson and Mohapatra \cite{PJ}
replaces the ordinary row sums by weighted row sums.

\begin{thm}[\cite{PJ}]\label{t2}
Let $M=[m_{ij}]$ be a $p\times q$ real matrix with constant column
sum $c$, and let $(x_1,\ldots,x_q)\in\mathbb{R}^q$ with
$s=\sum_{j=1}^q x_j>0$. Set
$d_i=\sum_{j=1}^q m_{ij}x_j$ and
$\bar d=\frac{1}{p}\sum_{i=1}^p d_i$.
If $f(t)=t\varphi(t)$ is convex on an interval containing
$d_1,\ldots,d_p$, then
\[
\frac{1}{s}\sum_{j=1}^q\sum_{i=1}^p
\varphi(d_i)m_{ij}x_j\geq c\varphi(\bar d).
\]
If $f$ is strictly convex, equality can occur only when
$d_1=\cdots=d_p$.
\end{thm}
The discrete framework was subsequently extended in \cite{johnson} to
pairs of positive measure spaces by replacing matrices and discrete
weights with measurable kernels and weight functions. We recall this
measure-theoretic version, which will be used throughout the paper. Throughout, a positive measure space means a measure space equipped
with a positive measure. For two such spaces $(V,\mu)$ and $(E,\tau)$,
we write $\mu\times\tau$ for the product measure. All functions and
kernels are assumed measurable whenever required, and all stated
integrals are assumed to be well defined.

\begin{thm} \cite{johnson} \label {t3}
Let $(V, \mu)$ and $(E, \tau)$ be positive measure spaces,\\
$M: V \times E \to \mathbb {R}$ a  $(\mu \times \tau)$-measurable function,\\
$wt: E \to \mathbb {R}$ a $\tau$-measurable function,\\
$I \subseteq \mathbb {R}$ an interval, and $\varphi: I \to \mathbb {R}$ a function such that $f(t) = t \varphi (t)$ is convex on $I$.  Suppose that the following, $(i)$ - $(vi)$, hold.

\begin{itemize}
\item [(i)] $0 < \mu (V) < \infty$.
\item [(ii)] $0 < s = \int _E wt \;\;\; d \tau < \infty$.
\item [(iii)] $\int _{V \times E} |M (v, e)| \mid wt (e)| d(\mu \times \tau) < \infty$.
\item [(iv)] For some real number $c$ and for almost all $v \in V$, 
$$
\int _V M(v, e) d \mu (v) = c,\quad \delta (v) = \int _E M(v, e) wt(e) d \tau (e) \in I,
$$
for almost all $e \in E$.

\item [(v)] $ \int_{V\times E} |\phi(\delta(v))|\,|M(v,e)|\,|wt(e)|\, d(\mu\times\tau) < \infty.$ With $\delta$ as defined in $(v)$, let $\bar {\delta} = \frac {1}{\mu (V)} \int _V \delta d_{\mu}$.
\end{itemize}
Then
$$  
\frac {1}{s} \int _{V \times E} \varphi (\delta (v)) M(v, e)  wt(e)  d(\mu \times \tau) \geq c \varphi (\bar {\delta}).  \qquad (*)
$$
Further, if $f$ is strictly convex on $I$ then equality holds in $(*)$ if and only if $\delta (v) = \bar {\delta}$ a.e. $(\mu)$.
\end{thm}

The principal contribution of the present work is the development of a
unified application framework for the measure-theoretic Jensen-type
inequality in Theorem~\ref{t3}. In particular, we derive two-sided
$L^p$-estimates and operator-order inequalities for positive integral
operators, obtain norm and entropy estimates for Fej\'er means, establish
probabilistic moment and quantitative variance bounds, and formulate
robustness estimates for random and deterministic erasures. These
applications demonstrate that the underlying measure-space inequality
provides a common mechanism connecting harmonic analysis, operator theory,
probability, and robustness problems.

For the convenience of the reader, we summarize below the principal notation used throughout
the paper.
\begin{table}[ht]
\centering
\caption{Notation used throughout the paper}
\label{tab:notation}
\small
\begin{tabular}{ll}
\hline
Notation & Meaning\\
\hline
$(V,\mu),(E,\tau)$ & Positive measure spaces\\
$\mu\times\tau$ & Product measure on $V\times E$\\
$M(v,e)$ & Measurable kernel in Theorem~\ref{t3}\\
$wt(e)$ & Weight function on $E$\\
$s$ & Total weight, $\displaystyle s=\int_E wt\,d\tau$\\
$c$ & Constant satisfying $\displaystyle\int_V M(v,e)\,d\mu(v)=c$\\
$\delta(v)$ & $\displaystyle\int_E M(v,e)wt(e)\,d\tau(e)$\\
$\bar{\delta}$ & $\displaystyle\mu(V)^{-1}\int_V\delta\,d\mu$\\
$K(v,e)$ & Nonnegative kernel of an integral operator\\
$T$ & Integral operator,
$\displaystyle Tg(v)=\int_EK(v,e)g(e)\,d\tau(e)$\\
$\mathbf{1}_V,\mathbf{1}_E$ & Constant-one functions on $V$ and $E$\\
$\|\cdot\|_p$ & $L^p$-norm on the relevant measure space\\
\hline
\end{tabular}
\end{table}

\section{Application to Product Measure Spaces}
The measure-theoretic framework of Theorem~\ref{t3} naturally leads to applications involving positive integral operators, probability, and erasure models.
Throughout this section, we write \[ Tg(v)=\int_E K(v,e)g(e)\,d\tau(e), \] where $K:V\times E\to[0,\infty)$ is measurable. In what follows, $T$ will be considered as an operator from $L^p(E,\tau)$ into $L^p(V,\mu)$ for various values of $p$, whenever the corresponding mapping is well defined. We impose, as needed, one or both of the normalization conditions

\begin{equation}\label{2.1}
\int_V K(v,e)\,d\mu(v)=1
\qquad\text{for $\tau$-a.e. }e,
\tag{1}
\end{equation}
and
\begin{equation}\label{2.2}
\int_E K(v,e)\,d\tau(e)=1
\qquad\text{for $\mu$-a.e. }v.
\tag{2}
\end{equation}

\subsection{Positive integral operators}

Condition \eqref{2.1} preserves total mass, while \eqref{2.2} gives
$T1_E=1_V$. When both (1) and (2) hold, we call $K$ and $T$
doubly stochastic. The following result gives two-sided $L^p$-bounds
for such operators.

\begin{thm}\label{thm:doubly-stochastic}
Let $0<\mu(V)<\infty$ and let $T$ be a doubly stochastic integral
operator. Then, for every $1\leq p<\infty$ and every nonnegative
$g\in L^p(E,\tau)$,
\[
\mu(V)^{\frac1p-1}\|g\|_1
\leq
\|Tg\|_p
\leq
\|g\|_p.
\]
Moreover, $T$ is an $L^p$-contraction for arbitrary
$g\in L^p(E,\tau)$. For $p>1$, equality in the lower bound holds if
and only if $Tg=\frac{\|g\|_1}{\mu(V)}
\quad \mu\text{-a.e.}$
\end{thm}

\begin{proof}
Let $1\leq p<\infty$. Since $t\mapsto |t|^p$ is convex and, by
\eqref{2.2}, $\int_E K(v,e)\,d\tau(e)=1,\,$
for $\mu$-a.e.\ $v\in V$, Jensen's inequality gives
\[
|Tg(v)|^p
=
\left|\int_E K(v,e)g(e)\,d\tau(e)\right|^p
\leq
\int_E K(v,e)|g(e)|^p\,d\tau(e).
\]
Integrating over $V$ and applying Fubini's theorem and \eqref{2.1}, we obtain
\[
\begin{aligned}
\|Tg\|_p^p \leq
\int_E |g(e)|^p
\left(\int_V K(v,e)\,d\mu(v)\right)d\tau(e) =\int_E|g(e)|^p\,d\tau(e)
=\|g\|_p^p.
\end{aligned}
\]
Hence $\|Tg\|_p\leq\|g\|_p$. For $p=\infty$, \eqref{2.2} gives
\[
|Tg(v)|
\leq
\|g\|_\infty\int_EK(v,e)\,d\tau(e)
=\|g\|_\infty,
\]
so $T$ is an $L^p$-contraction for every $1\leq p\leq\infty$.

Now let $g\geq0$ and $1\leq p<\infty$. Applying
Theorem~\ref{t3} with $M=K,\quad wt=g,\quad \varphi(t)=t^{p-1},$
we have $c=1$, $\delta=Tg$, and
$\bar{\delta}=\|g\|_1/\mu(V)$. Therefore,
\[
\|Tg\|_p
\geq
\mu(V)^{\frac1p-1}\|g\|_1.
\]
For $p>1$, strict convexity of $t^p$ shows that equality holds
precisely when $Tg=\frac{\|g\|_1}{\mu(V)}
\quad \mu\text{-a.e.}$

\end{proof}

\begin{rem}
If, in addition, $\mu(V)=1$, then Theorem~2.1 reduces to
\[
\|g\|_1\leq \|Tg\|_p\leq \|g\|_p,
\qquad g\geq0.
\]
\end{rem}

The mass-preserving condition \eqref{2.1} also yields an operator-order
estimate on $L^2$. For $u\in L^2(E,\tau)$, we use
$u\otimes u$ to denote the rank-one operator $(u\otimes u)g=\langle g,u\rangle u.$
The following result gives a lower bound for $T^*T$ in the Loewner
order.

\begin{prop}\label{prop:TstarT}
Let $(V,\mu)$ and $(E,\tau)$ be finite positive measure spaces, and
let $T:L^2(E,\tau)\to L^2(V,\mu)$ be a bounded integral operator
satisfying \eqref{2.1}. Then
\[
T^*\mathbf 1_V=\mathbf 1_E
\quad\text{and}\quad
T^*T\geq
\frac{1}{\mu(V)}
(\mathbf 1_E\otimes\mathbf 1_E).
\]
Equivalently, for every $g\in L^2(E,\tau)$,
\[
\|Tg\|_2^2
\geq
\frac{1}{\mu(V)}
\left|\int_E g\,d\tau\right|^2.
\]
\end{prop}

\begin{proof}
By Fubini's theorem and the normalization of $K$,
\[
\langle Tg,\mathbf 1_V\rangle
=
\int_Eg(e)\,d\tau(e)
=
\langle g,\mathbf 1_E\rangle,
\]
and hence $T^*\mathbf 1_V=\mathbf 1_E$. Therefore, by the
Cauchy--Schwarz inequality,
\[
|\langle g,\mathbf 1_E\rangle|^2
=
|\langle Tg,\mathbf 1_V\rangle|^2
\leq
\mu(V)\|Tg\|_2^2.
\]
Thus
\[
\|Tg\|_2^2
\geq
\frac{1}{\mu(V)}
|\langle g,\mathbf 1_E\rangle|^2
=
\frac{1}{\mu(V)}
\langle(\mathbf 1_E\otimes\mathbf 1_E)g,g\rangle,
\]
which is equivalent to
\[
T^*T\geq
\frac{1}{\mu(V)}
(\mathbf 1_E\otimes\mathbf 1_E).
\]
\end{proof}

\subsection{Application to Fej\'er means}
We now apply Theorem~\ref{t3} to the Fej\'er kernel, a classical
positive kernel in harmonic analysis. Let
$\mathbb{T}=\mathbb{R}/2\pi\mathbb{Z}$ be equipped with normalized
Lebesgue measure $dm(\theta)=d\theta/(2\pi)$, and let $F_N$ denote
the Fej\'er kernel. For $g\in L^1(\mathbb{T})$, define its $N$th
Fej\'er mean by
\[
\sigma_Ng(\theta)
=\int_{\mathbb{T}}F_N(\theta-\varphi)g(\varphi)\,dm(\varphi).
\]
The positivity and normalization of $F_N$ yield the following norm
and entropy estimates.

\begin{prop}
\label{prop:fejer}
Let $g\in L^1(\mathbb{T})$ be nonnegative and set $s=\int_{\mathbb{T}}g\,dm>0.$
Then, for every $1\leq p<\infty$, $\|\sigma_Ng\|_p\geq s.$ Moreover, if $\sigma_Ng>0$ a.e., then $\frac1s\int_{\mathbb{T}}\sigma_Ng
\log\left(\frac{\sigma_Ng}{s}\right)\,dm\geq0.$
For $p>1$, equality in the first inequality holds if and only if
$\sigma_Ng=s$ a.e.; the same equality condition holds for the
entropy inequality.
\end{prop}

\begin{proof}
Since $F_N\geq0$ and $\int_{\mathbb{T}}F_N(\theta-\varphi)\,dm(\theta)=1,$ apply Theorem~\ref{t3} with
\[
V=E=\mathbb{T},\qquad
M(\theta,\varphi)=F_N(\theta-\varphi),\qquad
wt(\varphi)=g(\varphi).
\]
Then $c=1$, $\delta(\theta)=\sigma_Ng(\theta)$, and, since
$m(\mathbb{T})=1$, Fubini's theorem gives
\[
\bar\delta=\int_{\mathbb{T}}\sigma_Ng\,dm
=\int_{\mathbb{T}}g\,dm=s.
\]
Taking $\varphi(t)=t^{p-1}$ in Theorem~\ref{t3} yields
\[
\frac1s\int_{\mathbb{T}}(\sigma_Ng)^p\,dm\geq s^{p-1},
\]
and hence $\|\sigma_Ng\|_p\geq s$. For the second assertion, take
$\varphi(t)=\log(t/s)$; since $t\log(t/s)$ is strictly convex on
$(0,\infty)$,
\[
\frac1s\int_{\mathbb{T}}\sigma_Ng
\log\left(\frac{\sigma_Ng}{s}\right)\,dm\geq0.
\]
The equality statements follow directly from the strict convexity
condition in Theorem~\ref{t3}.
\end{proof}

\begin{example}\label{ex:fejer}
Let $g(\theta)=1+a\cos(k\theta),\quad |a|<1,\quad k\in\mathbb{N}.$
Then $g>0$ and $\int_{\mathbb{T}}g\,dm=1$. Since the Fej\'er mean
multiplies the $k$th Fourier mode by
$1-\frac{k}{N+1}$ for $k\leq N$, we have
\[
\sigma_Ng(\theta)=
1+a\left(1-\frac{k}{N+1}\right)\cos(k\theta),
\qquad 1\leq k\leq N.
\]
Hence, by Theorem~\ref{prop:fejer}, for every $p\geq1$,
\[
\left[
\int_{\mathbb{T}}
\left(
1+a\left(1-\frac{k}{N+1}\right)\cos(k\theta)
\right)^p dm(\theta)
\right]^{1/p}\geq1,
\]
and
\[
\int_{\mathbb{T}}\sigma_Ng(\theta)
\log\bigl(\sigma_Ng(\theta)\bigr)\,dm(\theta)\geq0.
\]
For $p>1$, the first inequality is strict whenever $a\neq0$ and
$k\leq N$. If $k>N$, then $\sigma_Ng=1$, so equality holds in both
inequalities. Thus, equality corresponds to the Fej\'er mean
eliminating the nonconstant Fourier mode.
\end{example}

\subsection{Applications to probability and entropy}

The measure-theoretic framework of Theorem~\ref{t3} also admits a
natural probabilistic interpretation. After normalization, the
generalized degree function $\delta$ defines a probability density,
and Theorem~\ref{t3} yields moment and relative entropy inequalities.
We record the resulting estimates below.

\begin{prop}
\label{cor:probability-moment}
Assume the hypotheses of Theorem~\ref{t3} with
$\mu(V)=1$ and $cs=1$. Suppose that $\delta(v)\geq0$ for
$\mu$-almost every $v\in V$. Then, for every $p\geq1$,
\[
\int_V \delta(v)^p\,d\mu(v)\geq 1.
\]
If $p>1$, equality holds if and only if $\delta(v)=1$ for $\mu$-almost every $v\in V$.
\end{prop}

\begin{proof}
Since $\mu(V)=1$ and $cs=1$, we have
\[
\overline{\delta}
=
\int_V\delta(v)\,d\mu(v)
=
1.
\]
Choose $\phi(t)=t^{p-1}.$ Then $f(t)=t\phi(t)=t^p$
is convex on $[0,\infty)$ for $p\geq1$. By
Theorem~\ref{t3},
\[
\frac{1}{s}
\int_{V\times E}
\delta(v)^{p-1}M(v,e)\operatorname{wt}(e)
\,d(\mu\times\tau)
\geq c.
\]
Using $\delta(v)
=
\int_E M(v,e)\operatorname{wt}(e)\,d\tau(e),$ the left-hand side becomes $\frac{1}{s}\int_V\delta(v)^p\,d\mu(v).$ Since $cs=1$, multiplication by $s$ gives
\[
\int_V\delta(v)^p\,d\mu(v)\geq1.
\]
For $p>1$, strict convexity of $t^p$ gives the stated equality
condition.
\end{proof}

\begin{cor}
\label{prop:probability-entropy}
Assume the hypotheses of Theorem~\ref{t3}, with $\delta>0$ a.e., and define $q(v)=\frac{\delta(v)}{cs},
\quad u(v)=\frac{1}{\mu(V)}.$ Then $q$ and $u$ are probability densities on $(V,\mu)$, and
\[
D(q\|u):=\int_Vq(v)\log\frac{q(v)}{u(v)}\,d\mu(v)\geq0.
\]
Moreover, equality holds if and only if $q=u$ a.e., equivalently, $\delta(v)=\frac{cs}{\mu(V)}
\quad\text{for $\mu$-a.e. }v\in V.$

\end{cor}

\begin{proof}
Since $\int_V\delta\,d\mu=cs$, both $q$ and $u$ integrate to one.
Taking $\varphi(t)=\log\left(\frac{\mu(V)t}{cs}\right)$
in Theorem~\ref{t3}, where
$t\varphi(t)=t\log\!\left(\frac{\mu(V)t}{cs}\right)$ is strictly
convex, and using $\bar\delta=cs/\mu(V)$, we obtain
\[
\frac1s\int_V\delta(v)
\log\left(\frac{\mu(V)\delta(v)}{cs}\right)d\mu(v)\geq0.
\]
Since $q/u=\mu(V)\delta/(cs)$, division by $c$ gives
$D(q\|u)\geq0$. The equality condition in Theorem~\ref{t3} yields
$\delta=\bar\delta=cs/\mu(V)$ a.e., equivalently $q=u$ a.e.
\end{proof}

We next obtain a quantitative refinement in terms of variance. Define the
probability measure $\mathbb{P}$ on $V$ by
\[
d\mathbb{P}(v)=\frac{1}{\mu(V)}\,d\mu(v),
\]
and consider the random variable $X(v)=\delta(v).$
Then
\[
\mathbb{E}_{\mathbb{P}}[X]
=\frac{1}{\mu(V)}\int_V\delta(v)\,d\mu(v)
=\bar{\delta}
=\frac{cs}{\mu(V)},
\]
and
\[
\operatorname{Var}_{\mathbb{P}}(X)
=\frac{1}{\mu(V)}
\int_V\bigl(\delta(v)-\bar{\delta}\bigr)^2\,d\mu(v).
\]

\begin{prop}
\label{prop:variance-gap}
Assume the hypotheses of Theorem~\ref{t3}. Suppose that
$f(t)=t\phi(t)$ is twice continuously differentiable on an interval
containing the range of $\delta$ and that
\[
f''(t)\geq\alpha>0.
\]
Then
\[
\frac{1}{s}\int_{V\times E}
\phi(\delta(v))M(v,e)wt(e)\,d(\mu\times\tau)
-c\phi(\bar{\delta})
\geq
\frac{\alpha\mu(V)}{2s}
\operatorname{Var}_{\mathbb{P}}(X).
\]
\end{prop}

\begin{proof}
Since $f''\geq\alpha$, we have
\[
f(t)\geq
f(\bar{\delta})
+f'(\bar{\delta})(t-\bar{\delta})
+\frac{\alpha}{2}(t-\bar{\delta})^2.
\]
Taking $t=\delta(v)$ and integrating over $V$, the linear term vanishes
because $\int_V(\delta-\bar{\delta})\,d\mu=0.$
Hence
\[
\frac{1}{s}\int_V f(\delta)\,d\mu
-\frac{\mu(V)}{s}f(\bar{\delta})
\geq
\frac{\alpha}{2s}
\int_V(\delta-\bar{\delta})^2\,d\mu.
\]
Now, by the definition of $\delta$,
\[
\int_V f(\delta)\,d\mu
=
\int_{V\times E}
\phi(\delta(v))M(v,e)wt(e)\,d(\mu\times\tau),
\]
while $\frac{\mu(V)}{s}f(\bar{\delta})
=c\phi(\bar{\delta}).$
The result follows from
\[
\int_V(\delta-\bar{\delta})^2\,d\mu
=
\mu(V)\operatorname{Var}_{\mathbb{P}}(X).
\]
\end{proof}

\subsection{Applications to erasures and robustness}
We first consider independent random erasures. The following result
decomposes the expected reconstruction error into the squared error
caused by the mean erasure pattern and a fluctuation term due to
randomness.

Let $K_1,\ldots,K_N\in L^2(V,\mu)$ and
$g=(g_1,\ldots,g_N)\in\mathbb R^N$, and write
$Tg=\sum_{j=1}^N g_jK_j$. If the $j$th coefficient is independently
erased with probability $p_j$, let $\varepsilon_j$ be its erasure
indicator, so that $\mathbb P(\varepsilon_j=1)=p_j$ and
$\mathbb P(\varepsilon_j=0)=1-p_j$. The corresponding random erasure
error is $\mathcal E=\sum_{j=1}^N\varepsilon_jg_jK_j$.

\begin{thm} \label{thm:random-erasure-error}
Let $\mathcal E=\sum_{j=1}^N \varepsilon_j g_jK_j$, where the
$\varepsilon_j$ are independent Bernoulli random variables satisfying
$\mathbb P(\varepsilon_j=1)=p_j$ and
$\mathbb P(\varepsilon_j=0)=1-p_j$. Then
\[
\mathbb E\|\mathcal E\|_2^2
=
\left\|\sum_{j=1}^N p_jg_jK_j\right\|_2^2
+
\sum_{j=1}^N p_j(1-p_j)|g_j|^2\|K_j\|_2^2.
\]
\end{thm}

\begin{proof}
Write $\varepsilon_j=p_j+\xi_j$, where
$\mathbb E\xi_j=0$ and
$\mathbb E\xi_j^2=p_j(1-p_j)$. Then
\[
\mathcal E
=
\sum_{j=1}^Np_jg_jK_j
+
\sum_{j=1}^N\xi_jg_jK_j.
\]
Hence
\[
\mathbb E\|\mathcal E\|_2^2
=
\left\|\sum_{j=1}^Np_jg_jK_j\right\|_2^2
+
\mathbb E\left\|\sum_{j=1}^N\xi_jg_jK_j\right\|_2^2,
\]
because the mixed term has expectation zero. By independence,
$\mathbb E(\xi_i\xi_j)=0$ for $i\neq j$, and therefore
\[
\mathbb E\left\|\sum_{j=1}^N\xi_jg_jK_j\right\|_2^2
=
\sum_{j=1}^Np_j(1-p_j)|g_j|^2\|K_j\|_2^2.
\]
Combining the two identities proves the result.
\end{proof}

We next consider deterministic erasures. The following result shows
that the $L^p$-norm of the post-erasure output is bounded below solely
in terms of the mass retained after erasure.

Let $(V,\mu)$ and $(E,\tau)$ be positive measure spaces with
$0<\mu(V)<\infty$, and let
$Tg(v)=\int_EK(v,e)g(e)\,d\tau(e)$, where $K:V\times E\to[0,\infty)$
is measurable and satisfies $\int_VK(v,e)\,d\mu(v)=1$ for
$\tau$-a.e.\ $e\in E$. For an integrable $g:E\to[0,\infty)$ and a
measurable erased set $\Lambda\subseteq E$, set
$E_\Lambda=E\setminus\Lambda$, $s_\Lambda=\int_{E_\Lambda}g\,d\tau$,
and $T_\Lambda g(v)=\int_{E_\Lambda}K(v,e)g(e)\,d\tau(e)$.

\begin{prop} \label{prop:robust-lp-erasure}
With the above notation, if $s_\Lambda>0$, then for every $p\geq1$,
\[
\|T_\Lambda g\|_{L^p(V,\mu)}
\geq
\mu(V)^{\frac1p-1}s_\Lambda.
\]
For $p>1$, equality holds if and only if
$T_\Lambda g=s_\Lambda/\mu(V)$ $\mu$-a.e.
\end{prop}

\begin{proof}
Apply Theorem~\ref{t3} to the restricted edge space
$E_\Lambda$. The corresponding generalized degree function is
\[
\delta_\Lambda(v)
=
\int_{E_\Lambda}K(v,e)g(e)\,d\tau(e)
=
T_\Lambda g(v),
\]
and the constant column integral is $c=1$. Moreover, $\overline{\delta}_\Lambda
=
\frac{s_\Lambda}{\mu(V)}.$ Taking $\phi(t)=t^{p-1},$
we have $t\phi(t)=t^p$, which is convex for $p\geq1$. Hence
\[
\frac1{s_\Lambda}
\int_V(T_\Lambda g(v))^p\,d\mu(v)
\geq
\left(\frac{s_\Lambda}{\mu(V)}\right)^{p-1}.
\]
Therefore, $\|T_\Lambda g\|_p
\geq
\mu(V)^{\frac1p-1}s_\Lambda.$ The equality statement follows from strict convexity of $t^p$ for
$p>1$.
\end{proof}

\begin{cor} \label{cor:uniform-erasure-robustness}
Under the hypotheses of Proposition~\ref{prop:robust-lp-erasure},
let $\mathcal E$ be a family of admissible erased subsets of $E$.
Suppose that there exists $\eta\in(0,1]$ such that
$\int_{E\setminus\Lambda}g(e)\,d\tau(e)
\geq
\eta\int_Eg(e)\,d\tau(e),$
for every $\Lambda\in\mathcal E$. Then, for every $p\geq1$ and every
$\Lambda\in\mathcal E$,
\[
\|T_\Lambda g\|_{L^p(V,\mu)}
\geq
\eta\,
\mu(V)^{\frac1p-1}
\|g\|_{L^1(E,\tau)}.
\]
\end{cor}

\begin{proof}
For every $\Lambda\in\mathcal E$,\, $s_\Lambda
\geq
\eta\|g\|_{L^1(E,\tau)}.$ The assertion follows immediately from
Proposition~\ref{prop:robust-lp-erasure}.
\end{proof}

\section{Conclusion and future directions}

We have presented applications of the measure-theoretic inequality in
Theorem~\ref{t3} to positive integral operators, Fej\'er means,
probability and entropy, and robustness under erasures. The resulting
estimates illustrate connections between the underlying Jensen-type
inequality and operator theory, harmonic analysis, and probability.
Future work may consider broader classes of positive kernels,
quantitative stability estimates, and optimal systems minimizing
expected or worst-case error under random and deterministic erasures.

\section{Acknowledgment}

 The first author is grateful to the Mohapatra Family Foundation and the College of Graduate Studies at the University of Central Florida for supporting his postdoctoral fellowship during the course of this research.

\section{Funding}

This research did not receive any specific grant from funding agencies in the public, commercial, or not-for-profit sectors.

\end{document}